\documentclass[10pt]{article}

\usepackage{amsmath,graphicx}
\usepackage{adjustbox}
\usepackage{amssymb}
\usepackage{tikz, tikz-cd}
\usepackage{fullpage}
\usepackage[all]{xy}
\usepackage[margin=1in]{geometry}
\usepackage{amsthm}
\usepackage{amsfonts}
\usepackage{bbm}
\usepackage{eufrak}
\usepackage{comment}
\usepackage{amsmath}
\usepackage{amsthm}
\usepackage{bbm}
\usepackage{eufrak}
\usepackage{comment}
\usepackage{array} 
\usepackage{color}
\usepackage[utf8]{inputenc}
\usepackage[english]{babel}
\usepackage{hyperref}
\newcolumntype{L}{>{$}l<{$}}

\newtheorem{theorem}{Theorem}[section]
\newtheorem{lemma}[theorem]{Lemma}

\newtheorem{proposition}[theorem]{Proposition}
\theoremstyle{definition}

\newtheorem{remark} [theorem] {Remark}
\newtheorem{question} [theorem] {Question}

\newtheorem{thm}[theorem]{Theorem}

\theoremstyle{definition}

\newcommand{\C}{{\mathbb{C}}}
\newcommand{\F}{{\mathbb{F}}}

\newcommand{\Z}{{\mathbb{Z}}}

\newcommand{\lk}{\ell\text{k}}

\newcommand{\rank}{\text{rank}}

\DeclareMathOperator{\HFL}{HFL}
\DeclareMathOperator{\CFL}{CFL}
\DeclareMathOperator{\HFK}{HFK}
\DeclareMathOperator{\HF}{HF}
\DeclareMathOperator{\Kh}{Kh}
\DeclareMathOperator{\AKh}{AKh}
\DeclareMathOperator{\SFH}{SFH}
\DeclareMathOperator{\Mod}{Mod}

\begin{document}
	\title{Knot Floer homology, link Floer homology and link detection}
	\author{Fraser Binns and Gage Martin}

	\maketitle

\begin{abstract}
We give new link detection results for knot and link Floer homology inspired by recent work on Khovanov homology. We show that knot Floer homology detects $T(2,4)$, $T(2,6)$, $T(3,3)$, $L7n1$, and the link $T(2,2n)$ with the orientation of one component reversed. We show link Floer homology detects $T(2,2n)$ and $T(n,n)$, for all $n$. Additionally we identify infinitely many pairs of links such that both links in the pair are each detected by link Floer homology but have the same Khovanov homology and knot Floer homology. Finally, we use some of our knot Floer detection results to give topological applications of annular Khovanov homology. 
\end{abstract}

\begin{section}{Introduction}
Knot and link Floer homology are invariants of links in $S^3$~\cite{OS1,R,OS}. There are a number of formal similarities between these Floer theoretic invariants and the combinatorial theory Khovanov homology. Recently, Khovanov homology has been shown to detect a number of simple links~\cite{xie_classification_2019,G,xie_links_2020,li_two_2020,baldwin_khovanov_2020}. Some of these detection results have used knot and link Floer homology without going so far as to determine whether knot or link Floer homology detects the relevant link. Inspired by this work, we give such detection results for knot and link Floer homology. We remind the reader that the knot Floer homology of a link $L$ is computed using an associated knot, called the knotification of $L$, in a connected sum of $S^1 \times S^2$'s while the link Floer homology of $L$ is computed directly from the link $L$ in $S^3$.

Previously it was known that knot Floer homology detects the unknot~\cite{ozsvath_holomorphic_2004-1}, the trefoil~\cite{ghiggini2008knot}, the figure eight knot~\cite{ghiggini2008knot}, the Hopf link~\cite{ni2007knot,ozsvath_holomorphic_2004-1}, and the unlink~\cite{ni_homological_2014, hedden_geography_2018}. Link Floer homology was known to detect the trivial $n$-braid together with its braid axis~\cite{BG} and determine if a link is split~\cite{wang_link_2020}. It was also known that a stronger version of link Floer homology $\CFL^\infty$ detects the Borromean rings and the Whitehead link~\cite{gorsky_triple_2020}.

In this paper, we prove the following knot Floer homology detection results:

\newtheorem*{T24}{Theorem~\ref{knotFloerT24}}
\begin{T24}
    
    If $\widehat{\HFK}(L) \cong \widehat{\HFK}(T(2,4))$, then $L$ is isotopic to $T(2,4)$.

\end{T24}

Throughout, we take the links $T(2,2n)$ to be oriented as the closure of the $2$-braids $\sigma_1^{2n}$.

\newtheorem*{T26}{Theorem~\ref{HFKT26}}
\begin{T26}
   If $\widehat{\HFK}(L) \cong \widehat{\HFK}(T(2,6))$, then $L$ is isotopic to $T(2,6)$.
\end{T26}

Let $J_n$ be the link obtained from $T(2,2n)$ by reversing the orientation on one of the components. We have the following result:

\newtheorem*{JN}{Theorem~\ref{HFK:KN}}
\begin{JN}
    If $\widehat{\HFK}(L) \cong \widehat{\HFK}(J_n)$ for some $n$, then $L$ is isotopic to $J_n$.
\end{JN}

	\newtheorem*{T33}{Theorem~\ref{L6n1:HFK}}
\begin{T33}

If $\widehat{\HFK}(L)\cong\widehat{\HFK}(T(3,3))$ then $L$ is isotopic to $T(3,3)$.
\end{T33}

\newtheorem*{L7n1}{Theorem~\ref{HFK:L7n1}}
\begin{L7n1}
If $\widehat{\HFK}(L)\cong \widehat{\HFK}(L7n1)$, then $L$ is isotopic to $L7n1$.
\end{L7n1}

We also prove the following Link Floer homology detection results:

\newtheorem*{T2n}{Theorem~\ref{T(2,2n)}}
	\begin{T2n}
	    If $\widehat{\HFL}(L)\cong\widehat{\HFL}(T(2,2n))$ for some $n$, then $L$ is isotopic to $T(2,2n)$.
	\end{T2n}
	
	\newtheorem*{Tnn}{Theorem~\ref{HFL:T(n,n)}}
	\begin{Tnn}
If $\widehat{\HFL}(L)\cong \widehat{\HFL}(T(n,n))$, then $L$ is isotopic to $T(n,n)$.
\end{Tnn}

\newtheorem*{Hopf}{Proposition~\ref{HFLHopf}}
\begin{Hopf}
If knot Floer homology detects a link $L$, then link Floer homology detects $L \# H$ for each choice of component of $L$ to connect sum with.
\end{Hopf}

A consequence of these detection results is that every link currently known to be detected by Khovanov homology is also detected by either knot or link Floer homology. This leads to the following natural question;

\begin{question}\label{KhVsFloer}
Is there a link which Khovanov homology detects but which neither knot nor link Floer homology detects?
\end{question}

On the other hand, we show that there are infinitely many links detected by link Floer homology but which are detected by neither Khovanov homology nor knot Floer homology.

\newtheorem*{stronger}{Theorem~\ref{stronger}}
\begin{stronger}
   There exists infinitely many pairs of links $(L,L')$ such that link Floer homology detects $L$ and $L'$ but $\Kh(L) \cong \Kh(L')$ and $\widehat{\HFK}(L) \cong \widehat{\HFK}(L')$.
\end{stronger}

Finally, we use some of our torus link detection results to derive applications to annular Khovanov homology. Annular Khovanov homology is an invariant of links in the thickened annulus $A \times I$, sometimes thought of as $S^3 \setminus U$ where $U$ is an unknot or the annular axis. To do this we utilize a generalization of the Ozsv\'{a}th-Szab\'{o} spectral sequence, which relates annular Khovanov homology and knot Floer homology of the lift of the annular axis $\tilde{U}$ in $\Sigma(L)$, the double branched cover of $L$~\cite{roberts_knot_2013,grigsby_khovanov_2010}.

\newtheorem*{3Braid}{Theorem~\ref{detect}}
\begin{3Braid}
Let $L \subseteq A \times I \subseteq S^3$ be an annular link. If $\AKh( L , \Z / 2\Z) \cong \AKh( \widehat{\sigma_1 \sigma_2} , \Z / 2\Z)$ then $L$ is isotopic to $ \widehat{\sigma_1 \sigma_2}$  in $A \times I$.
\end{3Braid}

\newtheorem*{4Braid}{Theorem~\ref{thm:4braid}}
\begin{4Braid}
Let $L \subseteq A \times I \subseteq S^3$ be an annular link. If $\AKh( L , \Z / 2\Z) \cong \AKh( \widehat{\sigma_1 \sigma_2\sigma_3} , \Z / 2\Z)$ then $L$ is isotopic to $ \widehat{\sigma_1 \sigma_2\sigma_3}$ in $A \times I$.
\end{4Braid}

\newtheorem*{6Braid}{Theorem~\ref{thm:6braid}}
\begin{6Braid}
    Let $L \subseteq A \times I \subseteq S^3$ be an annular link. If $\AKh( L , \Z / 2\Z) \cong \AKh( \widehat{\sigma_1 \sigma_2\sigma_3\sigma_4 \sigma_5} , \Z / 2\Z)$ then $L$ is isotopic to $ \widehat{\sigma_1 \sigma_2\sigma_3\sigma_4 \sigma_5}$ in $A \times I$.
\end{6Braid}

Notice that $\widehat{\sigma_1 \sigma_2}$, $ \widehat{\sigma_1 \sigma_2\sigma_3}$, and $ \widehat{\sigma_1 \sigma_2\sigma_3\sigma_4 \sigma_5}$ all represent the unknot when considered in $S^3$ but these braid closures are all non-trivial knots in $A \times I$.

This paper is organised as follows;  in Section~\ref{Section:Preliminaries} we briefly review knot and link Floer homology. In section~\ref{Section:JN} we prove that knot Floer homology detects $J_n$ and link Floer homology detects $T(2,2n)$. In section~\ref{Section:T(2,4)} we prove that knot Floer homology detects $T(2,4)$. In section~\ref{SectionT(2,6)} we prove that knot Floer homology detects $T(2,6)$. In section~\ref{section T(n,n)} we prove that link Floer homology detects $T(n,n)$. In section~\ref{section:L6n1} we prove that knot Floer homology detects $T(3,3)$. In section~\ref{section:L7n1} we prove that knot Floer homology detects $L7n1$. In section~\ref{section:connectsum} we prove that there are infinite families of links dectected by link Floer homology that also have the same Khovanov homology and knot Floer homology. Finally, in section~\ref{section:Khovanov} we prove the annular Khovanov homology results using some of our knot Floer detection results.

\subsection*{Acknowledgements}
The authors would like to thank John Baldwin, Eli Grigsby, Siddhi Krishna, and Braeden Reinoso for helpful conversations. The authors would also like to thank Eugene Gorsky for communicating to us the topological argument in the the proof of Theorem~\ref{HFK:KN}, Tye Lidman for helpful comments on an earlier version of the paper and suggesting the arguments involving link symmetries in Section~\ref{section:Khovanov}, and Daniel Ruberman for helpful conversations about symmetries of Seifert fibered links. 

\end{section}

\begin{section}{Knot Floer homology and link Floer homology}\label{Section:Preliminaries}

Knot Floer homology and link Floer homology are invariants of links in $S^3$ defined using a version of Lagrangian Floer homology~\cite{OS1,R,OS}. They are categorifications of the single variable and multivariable Alexander polynomials respectively. Here we briefly highlight the key features of knot Floer homology and link Floer homology that we use to obtain our detection results. Throughout this paper we work with coefficients in $\mathbb{Z}/2\mathbb{Z}$.

 Let $L$ be an oriented link in $S^3$, with components $L_1,L_2,\dots, L_n$. The link Floer homology of $L$ is a multi-graded vector space
$$\widehat{\HFL}(L)=\underset{d,A_1,\dots A_n}{\bigoplus}\widehat{\HFL}_d(L;A_1,\dots A_n)$$
The grading denoted by ``$d$" above is called the Maslov or algebraic grading, while the $A_i$ gradings are called the Alexander gradings. Each $A_i$ satisfies $2A_i+\lk(L_i,L-L_i)\in 2\Z$.

The knot Floer homology of $L$ is a vector space bi-graded by a Maslov grading and a single Alexander grading. The knot Floer homology of $L$ can be obtained by projecting the link Floer homology complex onto the diagonal of the multi-Alexander gradings, which becomes the Alexander grading, and adding $\frac{n-1}{2}$ to the Maslov grading. 

 We use a number of formal properties of knot and link Floer homology in proving our link detection results. The first of these is that link Floer homology has a symmetry relating the component of the complex supported in grading $(m, A_1, \ldots , A_n)$ with the component of the complex supported in grading $(m - 2 \underset{i=1}{\overset{n}{\sum}} A_i , -A_1 , \ldots , - A_n)$. Knot Floer homology enjoys the same symmetry property, since it can be defined by projecting the multi-Alexander gradings onto the diagonal. There is also a K\"unneth formula for computing the link Floer homology of a connected sum in terms of a tensor product of link Floer homologies.

The main formal property we will use, however, is that the link Floer homology of $L$ admits spectral sequences to the link Floer homologies of its sublinks~\cite[Lemmas 2.2 and 2.3]{BG},\cite{OS}. In particular, when $L_i$ is a sublink of $L$ there is a spectral sequence from $\widehat{\HFL}(L)$ to $\widehat{\HFL}(L-L_i) \otimes V^{\vert L_i \vert} $ shifting each Alexander grading by $\frac{1}{2}\lk(L_j,L_i)$. It follows that there is also a spectral sequence from $\widehat{\HFL}(L)$ to $\widehat{\HF}(S^3)\otimes V^{n-1}$, or equivalently that there is a spectral sequence from $\widehat{\HFK}(L)$ to $\widehat{\HF}(\#^{n-1}(S^1\times S^2))$. Here $V$ is the multi-graded vector space $\F\oplus\F$ with non-zero Maslov gradings $0$ and $-1$ and multi-Alexander grading $(0,\ldots, 0)$.

In addition to enjoying the above algebraic properties, $\widehat{\HFK}(L)$ and $\widehat{\HFL}(L)$ are known to reflect a number of topological properties of $L$. For starters, there is a number of things we can say about the number of components of $L$. Since $\widehat{\HFL}(L)$ admits a spectral sequence to $V^{n-1}$, a link is a knot if and only if $\rank(\widehat{\HFK}(L))$ is odd. Since the Maslov grading for $\widehat{\HFL}(L)$ is integer valued, while the Maslov gradings of $\widehat{\HFL}(L)$ are $\Z+\frac{n-1}{2}$ valued, it follows that if $\widehat{\HFK}(L)$ has support contained in grading $\Z$ then $L$ has an odd number of components, while if $\widehat{\HFK}(L)$ has support contained in grading $\Z$ then $L$ has an even number of components. Finally, since $\widehat{\HFK}(L)$ admits a spectral sequence to ${\widehat{\HF}(\#^{n-1}(S^1\times S^2))}$ -- which has a generator of Maslov grading $\frac{n-1}{2}$ -- we have that $n\leq 1+2\max\{m:\rank(\widehat{\HFK}_m(L))\neq 0\}$.

Moreover, since knot Floer homology categorifies the Alexander-Conway polynomial, and the Alexander-Conway polynomial polynomial number of two component links by a result of Hoste \cite{hoste_first_1985}, it follows that knot Floer homology detects the linking number of two component links.

 We will also make use of the fact that the link Floer homology of $L$ yields information about the topology of $S^3 - L$; in particular that link Floer homology detects the Thurston norm of $S^3 - L$~\cite{ozsvath2008linkFloerThurstonnorm}. Finally, if $L$ is not a split link then the top Alexander grading associated to a component $L_i$ determines if $L - L_i$ is a braid in the complement of $L_i$. Specifically, this happens exactly when the rank in maximal non-zero Alexander grading is $2^{n-1}$~\cite[Proposition~1]{G}.

\end{section}

\begin{section}{Knot Floer homology detects $J_n$}\label{Section:JN}
	
	Given that knot Floer homology detects the genus of a link, it is natural to try and detect links of small genus. The one component case is that of the unknot, which knot Floer homology is known to detect. The two component case is that of the $2$-cable links. The simplest of these are $2$-cables of unknots, i.e. the links $T(2,2n)$ with the orientation of one component reversed. We call these links $J_n$. In this section we show that knot Floer homology detects each $J_n$.
	\begin{theorem}\label{HFK:KN}
		If $\widehat{\HFK}(L) \cong \widehat{\HFK}(J_n)$ for some $n$, then $L$ is isotopic to $J_n$.
	\end{theorem}
	
	For reference, we note that when $n$ is positive $\widehat{\HFK}(J_n)\cong \F^n_{\frac{3}{2}}[1]\oplus\F^{2n}_{\frac{1}{2}}[0]\oplus\F^n_{\frac{-1}{2}}[-1]$, where the subscript denotes the Maslov grading of the generator, and $[i]$ denotes the Alexander grading of a summand. When $n$ is negative then $\widehat{\HFK}(J_n)$ can be computed from the above formula using an understanding of how knot Floer homology is affected by mirroring.
	
	A consequence of this detection result is that link Floer homology detects the links $T(2,2n)$. This follows from Theorem~\ref{HFK:KN} by considering how link Floer homology changes under reversing the orientation of a single component.
	\begin{theorem}\label{T(2,2n)}
		If $\widehat{\HFL}(L)\cong\widehat{\HFL}(T(2,2n))$ for some $n$, then $L$ is isotopic to $T(2,2n)$.
	\end{theorem}
	
	To prove Theorem~\ref{HFK:KN}, we will first prove that $L$ is a 2-component link and that both of the components of $L$ are unknots. Then we will use the fact that knot Floer homology detects genus to show that $J_n$ is detected among 2-component links with unknotted components. The topological argument used here was communicated to the authors by Eugene Gorsky~\cite{GorskyPersonalCommunication} and also appears in B.~Liu's classification of the links $T(2,2n)$ in terms of surgery to a Heegaard Floer $L$-space~\cite{LiuLspace}.
	
	\begin{lemma}\label{Unknot:KN}
		If $\widehat{\HFK}(L)\cong\widehat{\HFK}(J_n))$ for some $n$, then $L$ is a 2-component link and both of the components are unknots.
	\end{lemma}
	\begin{proof}[Proof of Lemma~\ref{Unknot:KN}]
		First we show that $L$ is a 2-component link. Notice that the parity of the rank of $\widehat{\HFK}(L)$ rules out the case that $L$ is a knot. If $L$ is an $n$-component link then there is a spectral sequence from $\widehat{\HFK}(L)$ to $\widehat{\HF}(\#^{n-1}(S^1\times S^2))$. Because $\widehat{\HFK}(L)$ is only non-zero in Maslov gradings $-\frac{1}{2},\frac{1}{2},$ and $\frac{3}{2}$ this spectral sequence can only exist for $n=2$.
		
		To see that both components of $L$ are unknotted, we consider the spectral sequences from $\widehat{\HFK}(L)$ to $\widehat{\HFK}(K)\otimes V $ where $K$ is a component of $L$. From this spectral sequence we see that $\widehat{\HFK}(K)$ is zero in all Maslov gradings except possibly $0$ and $1$. Considering how the Maslov grading changes under the symmetry of the Alexander grading for knot Floer homology we can see that $\widehat{\HFK}(K)$ can only be supported in Alexander grading $0$, so $K$ is an unknot.
	\end{proof}
	With Lemma~\ref{Unknot:KN}, we are now ready to prove Theorem~\ref{HFK:KN}. The key step is to deduce that $J_n$ is a cable of the unknot.
	\begin{proof}[Proof of Theorem~\ref{HFK:KN}]
		
		Considering the Alexander grading of $\widehat{\HFK}(L)$, we see that the two components of $L$ bound an annulus. Thus $L$ is the twisted 2-cable of some knot. Each component of $L$ is isotopic to the knot that was cabled and so $L$ is a twisted 2-cable of the unknot. This means that $L$ is $J_m$ for some $m$. Finally, a simple computation of the respective ranks in each Maslov grading shows that $\widehat{\HFK}(J_m) \cong\widehat{\HFK}(J_n) $ if and only if $m= n$. Thus $L$ is isotopic to $J_n$.
		
	\end{proof}

\end{section}

\begin{section}{Knot Floer homology detects $T(2,4)$}\label{Section:T(2,4)}

Here we will utilize the results of the previous section to obtain a detection result for the torus link $T(2,4)$. The link Floer homology of $T(2,4)$ is shown in Table~\ref{HFL:T24}, for reference.

\begin{theorem}\label{knotFloerT24}
    
    If $\widehat{\HFK}(L) \cong \widehat{\HFK}(T(2,4))$, then $L$ is isotopoic to $T(2,4)$.

\end{theorem}

To prove this we show the following lemma:

\begin{lemma}\label{HFKT24unknots}
	If $\widehat{\HFK}(L) \cong \widehat{\HFK}(T(2,4))$ then $L$ consists of two components, $L_1,L_2$ such that each $\widehat{\HFK}(L_i)$ has a unique Maslov index $0$ generators. Moreover, that generator is supported in Alexander grading $0$ in $\widehat{\HFK}(L_i)$.
\end{lemma}

  We then show that $L$ has the same link Floer homology as $T(2,4)$ using structural properties of link Floer homology and apply Theorem~\ref{T(2,2n)} to complete the proof.

	\begin{table}[]\begin{center}

		\begin{tabular}{l|llll}
			$1$  &    & $\F_{-1}$  & $\F_0$  &  \\
			$0$  &  $\F_{-3}$   &  $\F_{-2}^2$  &  $\F_{-1}$  &  \\
			$-1$ & $\F_{-4}$    & $\F_{-3}$   &   &  \\ \hline
			& $-1 $& $0$ & $1$ & 
		\end{tabular}	\end{center}
	\caption{The link Floer homology of $T(2,4)$. The coordinates give the multi-Alexander grading, the subscript gives the Maslov grading.}\label{HFL:T24}
	\end{table}

The following lemma will be useful in proving Lemma~\ref{HFKT24unknots}:

\begin{lemma}\label{positiveAlex}
	Suppose $K$ is a component of a link $L$ such that $\widehat{\HFL}(L)$ is supported in Maslov gradings at most zero with a unique Maslov grading zero generator. Then there is a unique Maslov index grading $0$ generator in $\widehat{\HFK}(K)$, and it is of non-negative Alexander grading.
\end{lemma}

\begin{proof}
	The Maslov grading $0$ generator must persist under the spectral sequence from $\widehat{\HFL}(L)$ to $\widehat{\HFL}(K)\otimes V^{|L|-1}$, as else it cannot persist to $\widehat{\HF}(S^3) \otimes V^{|L|-1}$. If this generator sat in a negative Alexander grading then the symmetry properties of knot Floer homology would imply that there is a positive Maslov index generator in $\widehat{\HFL}(K)\otimes V^{|L|-2}$. However there are no positive Maslov index generators in $\widehat{\HFL}(L)$ and so there are none in $\widehat{\HFL}(K)\otimes V^{|L|-1}$.
\end{proof}
 
 With this in hand we can prove Lemma~\ref{HFKT24unknots}:

\begin{proof}[Proof of Lemma~\ref{HFKT24unknots}]

Suppose $L$ is an $n$ component link such that  $\widehat{\HFK}(L) \cong \widehat{\HFK}(T(2,4))$. Then $n\leq 2$ since $\widehat{\HFK}(L)$ admits a spectral sequence to $\widehat{\HF}(\#^{n-1}S^1\times S^2)$. Indeed, since $\rank(\widehat{\HFK}(L))$ is odd for knots, we have that $n=2$. Since knot Floer homology detects the linking number of two component links, it follows that $\lk(L_1,L_2)=2$.

	 There is only one generator in Maslov grading $0$ and it must survive in the spectral sequences from $\widehat{\HFL}(L)$ to $\widehat{\HFL}(L_i)\otimes V$. We call this generator $\theta_0$. The bi-Alexander grading for $\theta_0$ is then $(A_1 + \frac{l}{2} , A_2 + \frac{l}{2})$ where $A_i$ is the Alexander grading of the generator in Maslov grading $0$ in $\widehat{\HFK}(L_i)$ and $l$ is the linking number between the components. Since $A_1 + \frac{l}{2} + A_2 + \frac{l}{2} = 2$, and $l=2$, it follows that $A_1+A_2=0$. By Lemma~\ref{positiveAlex}, $A_1=A_2=0$, as desired.
\end{proof}

To complete our proof of Theorem~\ref{knotFloerT24} we show that if $L$ has the same knot Floer homology as $T(2,4)$ then $L$ also has same link Floer homology as $T(2,4)$. This result combined with Theorem~\ref{T(2,2n)} proves Theorem~\ref{knotFloerT24}.
	
	\begin{proof}[Proof of Theorem~\ref{knotFloerT24}]
	    Suppose $L$ is a link such that $ \widehat{\HFK}(L) \cong \widehat{\HFK}(T(2,4))$. We seek to understand $\widehat{\HFL}(L )$. From the argument in the proof of Lemma~\ref{HFKT24unknots} we know that the only Maslov grading $0$ generator of $\widehat{\HFL}(L )$ sits in bi-Alexander grading $(1,1)$.

	 Since there are spectral sequences from $\widehat{\HFL}(L)$ to $\widehat{\HFK}(K_i) \otimes V$ for each $i$, we see that there are also generators of $\widehat{\HFL}(L)$ in $(A_1,A_2)$ gradings $(1, 0)$ and $(0 , 1)$. The symmetry of $\widehat{\HFK}(L)$ gives generators at $(-1, -1)$, $(-1, 0)$ and $(0 , -1)$ as well. The are now two more generators to add so that the link Floer homology has rank $8$. To maintain an even rank in each $A_i$ grading, they both must be added at the same bi-grading. The only way to do this and maintain symmetry is to add them at $(0,0)$ so that $\widehat{\HFL}(L) \cong \widehat{\HFL}(T(2,4))$ and Theorem~\ref{T(2,2n)} shows $L$ is isotopic to $T(2,4)$.
	
		\end{proof}

\end{section}
\begin{section}{Knot Floer homology detects $T(2,6)$}\label{SectionT(2,6)}

In the previous section we showed that knot Floer homology detects the torus link $T(2,4)$. The torus link $T(2,6)$ is then a natural candidate for detection results. In this section we show that knot Floer homology indeed detects $T(2,6)$.

\begin{theorem}\label{HFKT26}
   If $\widehat{\HFK}(L) \cong \widehat{\HFK}(T(2,6))$, then $L$ is isotopic to $T(2,6)$.
\end{theorem}

Since $\widehat{\HFK}(L)$ admits a spectral sequence to $\widehat{\HF}(\#^{n-1}S^1\times S^2)$, where $n$ is the number of components of $L$, $L$ has at most two components. Indeed, as $\rank(\widehat{\HFK}(L))$ is even, $L$ has exactly two components. Since knot Floer homology detects the linking number of two component links, the linking number is three.

From here the proof of Theorem~\ref{HFKT26} amounts to an algebraic argument showing that $\widehat{\HFL}(L)\cong\widehat{\HFK}(L)$, and applying Theorem~\ref{T(2,2n)}.

For reference, after renormalizing the Maslov gradings to agree with the link Floer homology, the knot Floer homology of $T(2,6)$  is: rank one in $(M,A)$ gradings $(0,3)$ and $(-6  , -3)$; rank two in $(M,A)$ gradings $(-1,2)$, $(-2,1)$, $(-3,0)$, $(-4, -1)$, and $(-5 , -2)$; and rank zero in all other bi-gradings.

 \begin{proof}[Proof of Theorem~\ref{HFKT26}]
 
 Suppose that $L$ has the same knot Floer homology as $T(2,6)$. As in the proof that knot Floer homology detects $T(2,4)$ we have that $A_1+A_2+\lk(L_1,L_2)=3$, where each $A_i$ is the Alexander gradings of the Maslov index $0$ generator in $\widehat{\HFK}(L_i)$. Thus $A_1+A_2=0$, and Lemma~\ref{positiveAlex} implies that $A_1=A_2=0$.
 
 We now show that $L$ has the same link Floer Homology as $T(2,6)$, so it follows from Theorem~\ref{T(2,2n)} that $L$ is isotopic to $T(2,6)$.
 
 Since the linking number is $3$ and the Maslov index $0$ generator sits in Alexander grading $0$ in the knot Floer homology of each component, it follows that in $\widehat{\HFL}(L)$ the Maslov index $0$ generator sits in Alexander bi-grading $(\frac{3}{2},\frac{3}{2})$. The Maslov index $-1$ generators in $\widehat{\HFK}(L)$ must be in bi-Alexander gradings $(\frac{1}{2},\frac{3}{2}), (\frac{3}{2},\frac{1}{2})$. There is also Maslov index $-2$ generator in Alexander grading $(\frac{1}{2},\frac{1}{2})$. Consider the remaining Maslov index $-2$ generator. Suppose it sits in Alexander grading $(y,1-y)$. Observe that there must be Maslov index $-3$ generators sitting in Alexander grading $(y-1,1-y)$, $(y,-y)$. The symmetry property of $\widehat{\HFK}(L)$ then implies that $y=\frac{1}{2}$. The symmetry properties of $\widehat{\HFL}$ then show that $\widehat{\HFL}(L)\cong\widehat{\HFL}(T(2,6))$, and so by Theorem~\ref{T(2,2n)} we have that $L$ is isotopic to $T(2,6)$. as desired.

 \end{proof}

\end{section}

\begin{section}{Link Floer homology detects $T(n,n)$}\label{section T(n,n)}

In the previous section we showed that link Floer homology detects the $T(2,2n)$ torus links, motivated by detection results for $T(2,2)$, $T(2,4)$, and $T(2,6)$. The torus link $T(2,2)$ can also be viewed as one of the simplest links in the family of $T(n,n)$ torus links. In this section we show that link Floer homology detects the links $T(n,n)$. We use a characterization of $T(n+1,n+1)$ as an $(n)$-braid for $T(n, n)$ union the braid axis.

In~\cite{licata2012heegaard}, J.~Licata computes the link Floer homology of the links $T(n,n)$, aside from the Maslov gradings of certain generators when $n>6$. She conjectures a complete result. We prove that link Floer homology detects $T(n,n)$ using her computation. It follows from this that there are many graded vector spaces that do not arise as the link Floer homology of any link.

We will be interested in multi-graded vector spaces $B_n$ exhibiting the following four properties:

\begin{enumerate}
    \item There is a unique Maslov grading $0$ generator.
    \item The multi-Alexander grading of the Maslov grading $0$ generator is $(\frac{n-1}{2},\frac{n-1}{2},\dots,\frac{n-1}{2})$.
    \item $B_n$ has support contained only in multi-Alexander gradings $(A_1,A_2,\dots,A_n)$ satisfying $A_i\leq \frac{n-1}{2}$ for all $i$.
    \item $B_n$ has rank $2^{n-1}$ in $A_i$ grading $\frac{n-1}{2}$.
\end{enumerate}

In particular $\widehat{\HFL}(T(n,n))$ satisfies the above properties. Observe that if $L$ is any link whose link Floer homology satisfies all of the above conditions then $L$ is not a split link, so each component $L_i$ of $L$ is a braid axis for $L-L_i$.

\begin{theorem}\label{HFL:T(n,n)}
If $\widehat{\HFL}(L)\cong \widehat{\HFL}(T(n,n))$, then $L$ is isotopic to $T(n,n)$.
\end{theorem}

The main ingredient of this proof is a result stating that, under certain circumstances, if the link Floer homology of a link has certain algebraic properties then the linking numbers of certain components with the rest of the link are positive.

\begin{lemma}\label{positivelinking}
     Let $L$ be a link with components $L_i$ for $1\leq i\leq n$. Suppose that $\widehat{\HFL}(L)$ has a unique generator of Maslov index $0$ with $A_i$ grading $x\geq0$. Suppose $\widehat{\HFL}(L)$ is supported in $A_i$ gradings at most $x$. Then $\lk(L_i,L_j)\geq 0$ for all $j$.
\end{lemma}

\begin{proof}[Proof of Lemma~\ref{positivelinking}]

Let $\theta_0$ denote the unique Maslov index $0$ generator. The vector space $\widehat{\HF}(S^3) \otimes V^{n-1}$ is non-zero in Maslov grading $0$ so all other intermediate vector spaces with spectral sequences fitting between $\widehat{\HFL}(L)$ and $\widehat{\HF}(S^3) \otimes V^{n-1}$ must also be non-zero in this Maslov grading. Because $\theta_0$ is the only generator in this Maslov grading, it must survive in every such spectral sequence.

Consider the spectral sequence to $\widehat{\HFL}(L - L_j)\otimes V$ obtained by forgetting the component $L_j$. The $A_i$ grading on $\widehat{\HFL}(L - L_j)\otimes V$ will be shifted by $\frac{ \lk(L_i,L_j)}{2}$, we show that this shift must be non-negative. 

Because $\theta_0$ survives this spectral sequence, $\widehat{\HFL}(L - L_j)\otimes V$ will have top $A_i$ grading $\frac{n-1}{2}$. Considering the $A_i$ grading on $\widehat{\HFL}(L)$ we see that $\widehat{\HFL}(L - L_j)\otimes V$ will have bottom $A_i$ grading no smaller than $\frac{-n+1}{2}$. Since the $A_i$ grading on $\widehat{\HFL}(L - L_j)\otimes V$ must be symmetric about a non-negative number, the shift applied to the Alexander grading must be non-negative, so $\lk(L_i,L_j)\geq 0$.
\end{proof}

With this result on the non-negativity of linking numbers, we can proceed with the proof of Theorem~\ref{HFL:T(n,n)}. We will proceed by induction, using the characterization of $T(n+1,n+1)$ as the link consisting of the unique $n$-braid for $T(n,n)$ together with the braid axis.

\begin{proof}[Proof of Theorem~\ref{HFL:T(n,n)}]

Suppose that $\widehat{\HFL}(L)\cong \widehat{\HFL}(T(n,n))$. Lemma~\ref{positivelinking} tells us that $\lk(L_i,L_j)\geq 0$ for every distinct $i,j$. Moreover, because $L$ is not split and each component $L_i$ of $L$ is a braid axis for $L-L_i$ we have $\lk(L_i,L_j)\neq 0$. 

The top non-zero $A_i$ grading is $\frac{n-1}{2}$. The relationship between the top non-zero $A_i$ grading and the Seifert genus of  $L_i$ implies that; $$ \dfrac{n-1}{2}\geq g(L_i)+\underset{j\neq i}{\sum}\dfrac{\lk(L_i,L_j)}{2 }. $$

However, because $\lk(L_i,L_j) > 0$ we also have that $$g(L_i)+\underset{j\neq i}{\sum}\dfrac{\lk(L_i,L_j)}{2 } \geq g(L_i) + \frac{n-1}{2},$$ with equality when $\lk(L_i,L_j) = 1$ for all $j$. Combining these inequalities gives that $g(L_i)=0$, and $\lk(L_i,L_j)=1$ for all $i,j$.

 We now know that $L$ is an $n$-component link where each component is an unknot, each component is a braid axis for the rest of the link, and the linking number between any two components is $1$. The torus link $T(n,n)$ is the only $n$ component link satisfying all of these conditions. This can be verified by induction on $n$. Specifically, check explicitly that $T(2,2)$ is the only such $2$-component link, then view $L$ as a braid axis of some $n$-braid representing an $n$ component link satisfying the same properties.
\end{proof}

\end{section}

\begin{section}{Knot Floer homology detects $T(3,3)$}\label{section:L6n1}

In previous sections we showed that for some of the first members of the family of $T(2,2n)$ torus links the link Floer homology detection result can be strengthened to knot Floer homology detection results. In this section we do the same for $T(3,3)$, the third member of the $T(n,n)$ family.

\begin{theorem}\label{L6n1:HFK}
If $\widehat{\HFK}(L)\cong\widehat{\HFK}(T(3,3))$, then $L$ is isotopic to $T(3,3)$.
\end{theorem}

To prove we will use various spectral sequence arguments to show that $L$ has the same link Floer homology as $T(3,3)$. The above theorem then follows immediately from Theorem~\ref{HFL:T(n,n)}.

\begin{proof}[Proof of Theorem~\ref{L6n1:HFK}]
Suppose $L$ is an $n$ component link such that $\widehat{\HFK}(L)\cong\widehat{\HFK}(T(3,3))$.

First note that $n\leq 3$ as else $\widehat{\HFK}(L)$ would not admit a spectral sequence to $\widehat{\HF}(\#^{n-1}(S^1\times S^2))$. $n\neq 2$ since the Maslov gradings of $\widehat{\HFK}(L)$ are supported in integer gradings. Moreover, $L$ cannot be a knot since $\rank(\widehat{\HFK})(L)$ is even. Thus $n=3$.
   
 Let $L_1,L_2,L_3$ be the components of $L$. We now seek to determine the structure of $\widehat{\HFL}(L)$.
    
The symmetry of $\widehat{\HFL}(L)$ implies that the unique generator in Maslov grading $-2$ and Alexander grading $0$ sits in multi-Alexander grading $(0,0,0)$. Similarly the symmetry implies that at least one of the Maslov grading $-3$ generators also sits at multi-grading $(0,0,0)$  .

Since the Maslov grading $0$ generator in knot Floer homology is of Alexander grading, $3$, the Maslov grading $0$ generator in Link Floer homology sits in Alexander multi-grading $(x,y,3-x-y)$ for some pair of integers $(x,y)$. In order that the link Floer homology admits the requisite spectral sequences, there are Maslov grading $-1$ generators in multi-Alexander gradings $(x,y,2-x-y)$, $(x-1,y,3-x-y)$, $(x,y-1,3-x-y)$.

Now, observe that each Maslov grading $-1$ generator has at least one distinct Alexander grading from the unique Maslov grading $0$ generator. In order to admit the requisite spectral sequences there must be Maslov grading $-2$ generators with $(A_1,A_2)=(x,y-1),(x-1,y)$, $(A_1,A_3)=(x-1,3-x-y),(x,2-x-y)$ and $(A_2,A_3)=(y-1,3-x-y),(y,2-x-y)$. A direct computation shows that at most one of these corresponds to the generator in multi-grading $(0,0,0)$. Thus there are Maslov index $-2$ generators in Alexander gradings $(x-1,y-1,3-x-y)$, $(x,y-1,2-x-y)$ and $(x-1,y,2-x-y)$.

By a similar argument, we can see that that there is a Maslov index $-3$ generator in multi-Alexander grading $(x-1,y-1,2-x-y)$. If $(x,y)\neq (1,1)$ this determines the entire link Floer complex. If $x=y=1$ then the remaining Maslov index $-3$ generators must be of multi-Alexander grading $(0,0,0)$ to insure that each $(A_i,A_j)$ grading is of even rank, so again the entire complex is determined.
    
    Consider the Maslov index $-3$ generator in multi-Alexander grading $(0,0,0)$, $\theta_{-3}$. Since $\theta_{-3}$ does not persist under the spectral sequence to $\widehat{\HFL}(L_i)\otimes V^{\otimes 2}$ for any $i$, there must be a Maslov index $-2$ generator in non-zero Alexander grading with each Alexander grading at least $0$, or a Maslov index $-4$ generator with each Alexander grading at most $0$. Observe that these two conditions are equivalent by the symmetry of the complex. Thus we have that $x-1,y-1,3-x-y\geq 0$, $x,y-1,2-x-y\geq 0$, or $x-1,y-1,2-x-y\geq 0$. By permuting the components we may take $x-1,y-1,3-x-y\geq 0$. There are only three solutions $(x,y)=(1,1)$, $(x,y)=(2,1)$, or $(x,y)=(1,2)$. If $(x,y)=(1,1)$ then $\widehat{\HFL}(L)\cong\widehat{\HFL}(T(3,3))$.
    
    We complete the proof by excluding the cases $(x,y)=(2,1)$, $(x,y)=(1,2)$. After permuting components, we may take $x=1,y=2$ without loss of generality.
    
     Since $\rank(\widehat{\HFK}(L))=18$,  $\rank(\widehat{\HFK}(L_i))\leq \dfrac{9}{2}$. It follows that each component is an unknot or a trefoil. Observe that if a component is a trefoil then it must be $T(2,3)$, as there are no positive Maslov index generators in $\widehat{\HFK}(L)$. Indeed there can be no $T(2,3)$ component as this would require there to be an Alexander grading two less than an Alexandar grading of the Maslov index $0$ generator containing a summand $\F_{-2}\oplus\F_{-3}^2\oplus\F_{-4}$, which does not occur. Thus each component is an unknot. From here we can compute the linking numbers from the Alexander gradings of the Maslov grading $0$ generator. We find $\lk(L_1,L_3)=-1=-\lk(L_2,L_3)$, and $\lk(L_1,L_2)=3$. Since $L$ is not a split link, we see that $L-L_3$ is a $2$-braid in the complement of $L_3$, when one of the two strands is given the opposite orientation. Each of $L_1$ and $L_2$ are unknots and $\lk(L_1,L_2)=3$ so after changing the orientation of $L_1$, $L-L_3$ is the 2-braid $T(2,-6)$. However, the rank of knot Floer homology is invariant under changing orientations and $\rank(\widehat{\HFK}(T(2,6)\otimes V))=24$, so $T(2,6)$ cannot be a sublink of $L$ and $(x,y) \not = (1,2)$.
\end{proof}

\end{section}

\begin{section}{Knot Floer homology detects $L7n1$}\label{section:L7n1}
	
We have now shown that knot Floer homology detects a number of the low crossing number links that Khovanov homology is known to detect. In this section we continue this task, showing that knot Floer homology detects the link $L7n1$.

\begin{theorem}\label{HFK:L7n1}
If $\widehat{\HFK}(L)\cong \widehat{\HFK}(L7n1)$, then $L$ is isotopic to $L7n1$.
\end{theorem}

Our proof relies on the observation that $L7n1$ can be realized as a 2-braid representing $T(2,3)$ together with the braid axis.

First note that it follows from the fact that $\widehat{\HFK}(L)$ admits a spectral sequence to $\widehat{\HF}(\#^{n-1}(S^1\times S^2))$ -- where $n$ is the number of components of $L$ -- and the fact that the knot Floer homology of a knot is of odd rank that $L$ is a two component link. Since knot Floer homology detects the linking number of two component knots, it follows that the linking number of $L$ is two. From here we break up the proof of Theorem~\ref{HFK:L7n1} into the following lemmas;

\begin{lemma}\label{L7n1:HFK,HFL}
			Suppose $L$ is a two component link such that $\widehat{\HFK}(L7n1)\cong\widehat{\HFK}(L)$. Then $\widehat{\HFL}(L)\cong\widehat{\HFL}(L7n1)$.
		\end{lemma}
		
		\begin{lemma}\label{L7n1:HFL}
		     Suppose $L$ satisfies $\widehat{\HFL}(L)\cong\widehat{\HFL}(L7n1)$. Then $L$ is isotopic to $L7n1$. 
		\end{lemma}
		
The combination of these Lemmas immediately gives the proof of Theorem~\ref{HFK:L7n1}.

			L7n1 has homology as computed in \cite{OS}, and shown in Table~\ref{L7n1Table}.

	\begin{table}[]\begin{center}
	\begin{tabular}{l|lll}
		$2$    && $\F_{-1}$  & $\F_0$  \\
		$1$   &   &  $\F_{-2}$  &  $\F_{-1}$  \\
		$0$ &     & $\F_{-2}\oplus\F_{-3}$   &   \\
		$-1$ & $\F_{-5}$    & $\F_{-4}$   &   \\
		$-2$ & $\F_{-6}$    & $\F_{-5}$   &   \\
		 \hline
		& 	$-1$ & $0$&$1$ 
	\end{tabular}\end{center}
\caption{The link Floer homology of L7n1.}\label{L7n1Table}
\end{table}

 Lemma~\ref{L7n1:HFK,HFL} is proven by combining the symmetry and parity properties of the link Floer homology complex.

\begin{proof}[Proof of Lemma~\ref{L7n1:HFK,HFL}]
Since $L$ has two components, $\widehat{\HFL}(L)$ has exactly $2$ Alexander gradings.

Let $\theta_0$ be the Maslov grading $0$ generator. This generator $\theta_0$ has bi-Alexander grading $(\frac{3}{2}+x,\frac{3}{2}-x)$ for some $x$. Indeed, there must be generators sitting in gradings $(\frac{1}{2}+x,\frac{3}{2}-x)$, $(\frac{3}{2}+x,\frac{1}{2}-x)$ each of Maslov index $-1$. Together with the symmetry properties of link Floer homology, this determines the Alexander bi-gradings of $6$. The same symmetry properties also imply that the two generators in Alexander grading $0$ must have bi-Alexander grading $(0,0)$. Thus, up to choice of $x$, we need only specify the location of one more generator to determine the whole link Floer homology complex. Since each Alexander grading to be of even rank, the remaining Maslov grading $-2$ element must be in bi-Alexander grading $(\frac{1}{2}+x,\frac{1}{2}-x)$. Moreover, since the Maslov grading $-3$ component, $\theta_{-3}$, cannot persist in the spectral sequence to $\widehat{\HF}(S^3) \otimes V $, it follows that $x\in\{\frac{1}{2},0,-\frac{1}{2}\}$, for the Alexander gradings obstruct the existence of generators $y$ with $\langle\partial y,\theta_{-3}\rangle\neq 0$, and $\partial\theta_{-3}$ can only be non-zero if $x$ is in this range. Indeed, $x\neq 0$, since otherwise we would have an element with Alexander grading in $\Z$, and another with Alexander grading in $\Z+\frac{1}{2}$. The remaining two possibilities give link Floer homologies that agree with $L7n1$, as desired.

\end{proof}

We complete the proof of Theorem~\ref{HFK:L7n1} by showing that link Floer homology detects the link $L7n1$. We use the fact that $L7n1$ is the closure of a braid for $T(2,3)$ together with its braid axis.

		\begin{proof}[Proof of Lemma~\ref{L7n1:HFL}]
			Suppose a link $L$, with components $L_1$ and $L_2$ satisfies $\widehat{\HFL}(L)\cong\widehat{\HFL}(L7n1)$.
			
			Observe that $\rank(\widehat{\HFK}(L_i))\leq 5$, with equality if and only if the spectral sequence corresponding to $L_i$ collapses on the $E_1$ page. If the spectral sequence collapses on the $E_1$ page, then $\widehat{\HFK}(L_i)$ would have no shift applied to its Alexander grading as it is already symmetric around grading $0$. Therefore, if $\rank(\widehat{\HFK}(L_i))= 5$ then  $\lk(L_1,L_2)=0$. However, this is impossible because $L$ is not split and so the link Floer homology shows that $L-L_i$ is braided with respect to $L_i$.
			
			Thus $\rank(\widehat{\HFK}(L_i))<5$ and the link has components that are either unknots and trefoils. Observe that any trefoil component must be $T(2,3)$, as there are no generators of positive Maslov grading.
			
			Suppose $L_1,L_2$ are both unknots. The shifts in Alexander grading coming from the spectral sequences from $\widehat{\HFL}(L)$ to $\widehat{\HFL}(L_i)\otimes V$ tells us that $\lk(L_1,L_2)=4$ while $\lk(L_2,L_1)=2$, a contradiction.
			
			Observe that $L_1$ cannot be a trefoil for there is no Maslov grading $-2$ generator in an $A_1$ grading two less than the $A_1$ grading of the unique grading $0$ element. Thus $L_2$ is $T(2,3)$ and $\lk(L_2,L_1)=2$. Since $L_2$ is a 2-braid closure in the complement of $L_1$ and $L_2$ is $T(2,3)$, the link $L$ must be L7n1.
		\end{proof}

	\end{section}

\begin{section}{Connected sums with a Hopf link}\label{section:connectsum}

In this section we deduce some properties of link Floer homology under the operation of taking a connected sums with a Hopf link. We then explore some applications of these properties to the question of link detection. Our main application is to find two infinite families of links which are not detected by Khovanov homology or knot Floer homology but which are detected by link Floer homology. Throughout this section we let $H$ denote the Hopf link.

\begin{proposition}
A link $L$ can be expressed as $L' \# H$ if and only if there is an Alexander grading in $\widehat{\HFL}(L)$ where the span of its non-zero grading levels is $\{ -1/2 , 1/2\}$.
\end{proposition}

\begin{proof}
    This observation follows directly from the connection between link Floer homology and the Thurston norm. A component has a span of its non-zero grading levels $\{ -1/2 , 1/2\}$ if and only if that component bounds a disk which intersects the rest of the link in a single point. This is equivalent to expressing $L$ as $L' \# H$.
\end{proof}

This observation has the following immediate consequence;

\begin{proposition}\label{HFLHopf}
If knot Floer homology detects a link $L$, then link Floer homology detects $L \# H$ for each choice of component of $L$ to connect sum with.
\end{proposition}

\begin{proof}
    Considering the grading associated to the new unknotted component and then collapsing all of the the other Alexander gradings distinguishes that the link is $L \# H$ for some choice of component to connect sum onto. Comparing the multi gradings associated to the components of $L$ on $\widehat{\HFL}(L \#H)$ and those on $\widehat{\HFL}(L)$ determines which component of $L$ was chosen for the connected sum.
\end{proof}

Combining Proposition~\ref{HFLHopf} and previous knot Floer detection results immediately gives some new links that are detected by link Floer homology.

We now provide two infinite families of links which are detected by link Floer homology but are not detected by Khovanov homology or knot Floer homology.

\begin{theorem}~\label{stronger}
   There exists infinitely many pairs of links $(L,L')$ such that link Floer homology detects $L$ and $L'$ but $\Kh(L) \cong \Kh(L')$ and $\widehat{\HFK}(L) \cong \widehat{\HFK}(L')$.
\end{theorem}

To prove Theorem~\ref{stronger} we will first introduce two families of links and show every link in the families are detected by link Floer homology. This is the content of Theorems~\ref{family1} and~\ref{family2}. Then we will highlight explicit examples within these families which cannot be distinguished using Khovanov homology or knot Floer homology.

Both families of links are trees of unknots. We describe each family here. For the first family, let $L_n$ be the tree of unknots corresponding to the graph with $n-1$ vertices each connected to a single vertex. For the second family, let $L_{(a,b)}$ be the tree of unknots corresponding to the graph with $a+b + 2$ vertices with $a$ vertices connected to a vertex $x$, $b$ vertices connected to a vertex $y$ and an edge connecting the vertex $x$ and the vertex $y$.

\begin{theorem}~\label{family1}
For each $n \geq 2$, if $\widehat{\HFL}(L) \cong \widehat{\HFL}(L_n)$, then $L$ is isotopic to the link $L_n$.
\end{theorem}

\begin{remark}
    The links $L_n$ can be viewed as the trivial $n-1$-braid together with its braid axis. So Theorem~\ref{family1} was already known because link Floer homology detects braid closures~\cite[Proposition~1]{G} and detects the trivial braid amongst braid closures~\cite[Theorem~3.1]{BG}. However, we provide a different proof of Theorem~\ref{family1} because it is a simpler case of the ideas used in the proof of Theorem~\ref{family2}.
\end{remark}

\begin{proof}[Proof of Theorem~\ref{family1}]

   Suppose $L$ has the same link Floer homology as $L_n$.  First notice that $L$ cannot be a split link because there is a generator with all of its Alexander gradings non-zero. Additionally we see that $n-1$ components of $L$ each bound a disk intersecting the rest of $L$ in exactly one point. By the observation that $L$ is not split, we see that each of these $n-1$ components must bound a disk which only intersects the final component of $L$. So then $L$ is isotopic to the link $L_n$.
\end{proof}

\begin{theorem}~\label{family2}
For every pair $(a,b)$ with $a$ and $b$ positive, if $\widehat{\HFL}(L) \cong \widehat{\HFL}(L_{(a,b)})$, then $L$ is isotopoic to the link $L_{(a,b)}$.
\end{theorem}

\begin{proof}

First notice that link Floer homology detects the link $L_{(0,b)} = L_{b+1}$. We will now proceed by induction on $a$.

    Suppose that $L$ has the same link Floer homology as $L_{(a,b)}$. First notice that $L$ cannot be a split link because there is a generator with all of its Alexander gradings non-zero. Additionally we see that $a+b$ components of $L$ each bound a disk intersecting the rest of $L$ in exactly one point. By the observation that $L$ is not split, we see that each of these $a+b$ components must bound a disk which only intersects one of the final two components of $L$. Call these final components $X$ and $Y$ based on if their Alexander gradings agree with the Alexander gradings associated to the component in the tree of unknots for the vertex $x$ or $y$ respectively. Without loss of generality, at least one component bounds a disk that intersects $X$ in a single point. Then $L$ can be written as $L' \# H$ where the connect sum is taken along the component $X$. A quick computation shows that $\widehat{\HFL}(L') \cong \widehat{\HFL}(L_{(a-1, b)})$.  By induction $L'$ is isotopic to $L_{(a-1, b)}$ whence $L$ is isotopic to $L_{(a, b)}$.
\end{proof}

With these detection results in place, we are now ready to prove Theorem~\ref{stronger}.

\begin{proof}[Proof of Theorem~\ref{stronger}]

Consider the links $L_n$ and $L_{(a,b)}$ with $a + b + 1 = n$. These links are detected by link Floer homology. We now check that $\Kh(L_n) \cong \Kh(L_{(a,b)})$ and $\widehat{\HFK}(L_n) \cong \widehat{\HFK}(L_{(a,b)})$.

Both links can be constructed by starting with an unknot and connect summing a Hopf link $n$ times in total. A simple computation shows that Khovanov homology and knot Floer homology of $L \# H$ does not depend on which component of $L$ the Hopf link is connect summed onto. This shows $\Kh(L_n) \cong \Kh(L_{(a,b)})$ and $\widehat{\HFK}(L_n) \cong \widehat{\HFK}(L_{(a,b)})$.
    
\end{proof}

\end{section}

	\begin{section}{Applications to annular Khovanov homology}\label{section:Khovanov}
	
	Annular Khovanov homology was defined by Asaeda-Przytycki-Sikora~\cite{asaeda_categorification_2004} as a categorification of the Kauffman bracket skein module of the thickened annulus. The resulting theory is an invariant of links in the thickened annulus $A\times I$ or alternatively the complement of an unknot in the 3-sphere $S^3 \setminus U$. In particular, annular Khovanov homology is well suited to studying braid closures~\cite{BG, grigsby_sutured_2014,hubbard_sutured_2017,hubbard_annular_2016}.
	
	In this section we apply some of our earlier knot Floer detection results to show that annular Khovanov homology detects certain braid closures. The proofs will rely on the spectral sequence from annular Khovanov homology of a link $L$ to the knot Floer homology of the lift of the annular axis in the $\Sigma(L)$~\cite{grigsby_sutured_2014,roberts_knot_2013}.
	
	We use knot Floer detection results for $T(2,3)$, $T(2,4)$ and $T(2,6)$ to show annular Khovanov homology detects the closure of the braids $\sigma_1 \sigma_2$, $\sigma_1 \sigma_2\sigma_3$, and $\sigma_1 \sigma_2\sigma_3 \sigma_4 \sigma_5$ .The structure of each proof is similar; first we use properties of annular Khovanov homology to deduce necessary topological properties of the annular knot like braidedness or unknottedness. Then we use a knot Floer detection result to show that the lift of the annular axis is $T(2,3)$, $T(2,4)$ or $T(2,6)$ respectively. Finally we translate this into information about the annular link.
	
	The spectral sequence from the annular Khovanov homology of an annular link $L$ to the knot Floer homology of the lift of the annular axis in $\Sigma(L)$ is defined with $\Z/2\Z$ coefficients, at times however we will work with annular Khovanov homology over $\C$ because with these coefficients annular Khovanov homology has the structure of an $\mathfrak{sl}_2(\C)$ representation~\cite[Proposition~3]{grigsby_annular_2018}.
	
	More is known about the knot Floer homology of genus~1 fibered knots so we are able to prove more results for the closure of the 3-braid $\sigma_1 \sigma_2$ than for the other two annular knots.

\begin{thm}~\label{rank6}
If $L$ is a 3-braid closure and $\dim(\AKh( L , \Z / 2\Z)) = 6$ then $L$ is isotopic to $ \widehat{\sigma_1 \sigma_2}$ or $\widehat{\sigma_1^{-1} \sigma_2^{-1}} $ in $A \times I$.
\end{thm}

\begin{proof}

The lift $\widetilde{U_L}$ of the braid axis $U$ in $\Sigma(L)$ is a genus 1 fibered knot.

The manifold $\Sigma(L) \setminus \widetilde{U_L}$ is naturally a sutured manifold where the sutures on $S^3 \setminus \widetilde{U_L}$ are two distinct pairs of meridional sutures lifted into the double branched cover from the product sutures on $A \times I$. There is a spectral sequence from $\AKh(L, \Z / 2 \Z)$ to $\widehat{\SFH}(-\Sigma(L) \setminus \widetilde{U_L}, \mathbb{Z}/2\mathbb{Z}) \cong \widehat{\HFK}(\widetilde{U_L}, -\Sigma(K), \mathbb{Z}/2\mathbb{Z}) \otimes V$ where $V$ is a 2-dimensional vector space supported in bi-gradings $(0,0)$ and $(-1,-1)$. Furthermore, the $k$-grading in $\AKh$ corresponds to the Alexander grading on $\widehat{\SFH}$ or $\widehat{\HFK}$~\cite[Theorem 2.1]{grigsby_sutured_2014}~\cite[Theorem 1.1]{roberts_knot_2013}.

From this spectral sequence we can see that $\widehat{\HFK}(\widetilde{U_K},-\Sigma(K), \mathbb{Z}/2\mathbb{Z}) $ has rank no larger than 3. Every genus 1 fibered knot has knot Floer homology at least rank 3 and there are only four genus 1 fibered knots with rank 3 knot Floer homologies and they are the left and right handed trefoils in $S^3$ and two knots in the Poincar\'e Homology Sphere~\cite[Corollary~1.6]{baldwin_note_2018}.

The monodromies of fibered knots are unique up to conjugation. The monodromy of a fibered knot in $\Sigma(L)$ is the image of a braid representing $L$ in $\Mod(S^1_1)$ under the Birman-Hilden correspondence. Finally because $B_3 \cong \Mod(S^1_1)$ conjugate monodromies must come from conjugate braids so $L$ must be isotopic to the closure of on of the 4 braids on this list that correspond these four possible fibered knots.

\begin{enumerate}
\item $\sigma_1 \sigma_2$
\item $\sigma_1^{-1} \sigma_2^{-1}$
\item $(\sigma_1 \sigma_2)^{-6}\sigma_1\sigma_2$
\item $(\sigma_1 \sigma_2)^{6}\sigma_1^{-1}\sigma_2^{-1}$
\end{enumerate}

A computation shows that the ranks of the Annular Khovanov homologies of the last two braid closures are larger than six~\cite{AKhMathematica}.

Therefore $L$ is isotopic to $ \widehat{\sigma_1 \sigma_2}$ or $\widehat{\sigma_1^{-1} \sigma_2^{-1}} $ in $A \times I$.

\end{proof}

The detection result in Theorem~\ref{detect} follows immediately from Theorem~\ref{rank6} and previous results about annular Khovanov homology.

\begin{thm}~\label{detect}
Let $L \subseteq A \times I \subseteq S^3$ be an annular link. If $\AKh( L , \Z / 2\Z) \cong \AKh( \widehat{\sigma_1 \sigma_2} , \Z / 2\Z)$ then $L$ is isotopic to $ \widehat{\sigma_1 \sigma_2}$  in $A \times I$.
\end{thm}

\begin{proof}

If $\AKh( L , \Z / 2\Z) \cong \AKh( \widehat{\sigma_1 \sigma_2} , \Z / 2\Z)$ then $L$ is isotopic to a 3-braid closure~\cite[Corollary~1.2]{grigsby_sutured_2014}~\cite[Corollary~8.4]{xie_instanton_2019}.

From Theorem~\ref{rank6} we have that $L$ is isotopic to $ \widehat{\sigma_1 \sigma_2}$ or $\widehat{\sigma_1^{-1} \sigma_2^{-1}} $. A simple computation shows that $\AKh( \widehat{\sigma_1^{-1} \sigma_2^{-1}} , \Z / 2\Z)  \not \cong \AKh( \widehat{\sigma_1 \sigma_2} , \Z / 2\Z)$ so $L$ must be isotopic to $ \widehat{\sigma_1 \sigma_2}$  in $A \times I$.
\end{proof}

One interpretation of Theorem~\ref{rank6} is that $ \widehat{\sigma_1 \sigma_2}$ and $\widehat{\sigma_1^{-1} \sigma_2^{-1}} $ are the simplest 3-braids from the point of view of annular Khovanov homology. 

\begin{proposition}\label{minrank}
If $L$ is isotopic to a 3-braid closure in $A \times I$ then  $\dim(\AKh( L , \Z / 2\Z)) \geq  6$.

\end{proposition}

\begin{proof}
From the universal coefficient theorem it follows that $\dim(\AKh( L , \Z / 2\Z)) \geq \dim(\AKh( L , \C))$ so it suffices to show that $\dim(\AKh( L , \C))\geq 6$. Because $L$ is a 3-braid closure, $\AKh( L , \C)$ has dimension one in grading $k = 3$. The $\mathfrak{sl}_2(\C)$ action on $\AKh( L , \C)$ and the fact that the $k$-grading gives the $\mathfrak{sl}_2(\C)$ weights implies that $\AKh( L , \C)$ contains an irreducible weight 3 representation of $\mathfrak{sl}_2(\C)$ showing that $\dim(\AKh( L , \Z / 2\Z)) \geq 4$.

The $\mathfrak{sl}_2(\C)$ action gives a symmetry in the $k$ gradings. Since $\AKh( L , \C)$ is zero in $k = 0$, the dimension of $\AKh( L , \C)$ must be even. It thus only remains to rule out the case that $\dim(\AKh( L , \Z / 2\Z)) = 4$.

If $\dim(\AKh( L , \Z / 2\Z)) = 4$ then $\AKh( L , \C)$ consists only of an irreducible weight 3 representation of $\mathfrak{sl}_2(\C)$ which must live in a single homological grading. Because $\AKh( L , \C)$ is supported in a single homological grading, the spectral sequence from $\AKh( L , \C)$ to $\Kh(L , \C)$ collapses. The proof of Theorem~3.1(a) in~\cite{BG} shows that the only braid closures for which this spectral sequence collapses immediately are closures of trivial braids. A computation shows that $\dim(\AKh( \widehat{1}_3 , \C)) > 6$ so there is no 3-braid closure with $\dim(\AKh( L , \Z / 2\Z)) = 4$.
\end{proof}

We now show that annular Khovanov homology detects the closure of the 4-braid $\sigma_1 \sigma_2\sigma_3$ in $A \times I$. We will first show that any braid closure with the same annular Khovaonov homology as $\widehat{\sigma_1 \sigma_2\sigma_3}$ must represent an unknot in $S^3$. We then show that the lifted braid axis is $T(2,4)$ using knot Floer homology. Then we complete the proof by understanding the symmetries of $T(2,4)$.

\begin{theorem}\label{thm:4braid}
    Let $L \subseteq A \times I \subseteq S^3$ be an annular link. If $\AKh( L , \Z / 2\Z) \cong \AKh( \widehat{\sigma_1 \sigma_2\sigma_3} , \Z / 2\Z)$ then $L$ is isotopic to $\widehat{\sigma_1\sigma_2\sigma_3}$ in $A \times I$.
\end{theorem}

We break the proof of Theorem~\ref{thm:4braid} up by first proving the following two lemmas.

\begin{lemma}\label{4unknot}
     Let $L \subseteq A \times I \subseteq S^3$ be an annular link. If $\AKh( L , \Z / 2\Z) \cong \AKh( \widehat{\sigma_1 \sigma_2\sigma_3} , \Z / 2\Z)$ then $L$ is an unknot in $S^3$.
\end{lemma}

\begin{proof}
    We first compute $\AKh( L , \mathbb{C})$ from $\AKh( L , \Z / 2\Z)$. Throughout, we will use that the dimension of annular Khovanov homology over $\mathbb{C}$ can be no larger than that over $\Z / 2\Z$. Because $L$ is a 4-braid closure, $\AKh( L , \mathbb{C})$ must contain a weight 4 irreducible $\mathfrak{sl}_2(\mathbb{C})$ representation in grading $i=0$, of dimension $5$. Thus $\AKh( L , \mathbb{C})$ must consist only of this representation in grading $i=0$ because $\AKh( L , \Z / 2\Z)$ has dimension $5$ in homological grading $0$.  We therefore have that all of the generators in grading $i = 1$ for $\AKh( L , \Z / 2\Z)$ must correspond to generators of $\AKh( L , \mathbb{C})$, for it not, they would correspond to $2$-torsion in $\AKh( L , \mathbb{Z})$, but the torsion contributes dimension in two different homological gradings by the universal coefficient theorem.

    A simple computation of annular Khovanov homology verifies that $L$ is not the trivial braid. Thus by~\cite[Theorem~3.1]{BG}, we know that the differential $\partial_-$ on $\AKh( L , \mathbb{C})$ inducing the spectral sequence to $\Kh(L)$ must send the highest weight generator in the grading $i = 0$ to the highest weight generator in the grading $i = 1$. The action $\partial_-$ is part of the action of $\mathfrak{sl}_2(\wedge)$ on $\AKh( L , \mathbb{C})$ and commutes up to sign with the lowering operator $f$~\cite[Theorem~1]{grigsby_annular_2018}. This means that the image of $\partial_-$ is spanned by all generators in grading $i = 1$. Thus $\Kh(L)$ is dimension $2$, whence $L$ is the unknot.
\end{proof}

\begin{lemma}\label{4HFK}
        Let $L \subseteq A \times I \subseteq S^3$ be an annular link. If $\AKh( L , \Z / 2\Z) \cong \AKh( \widehat{\sigma_1 \sigma_2\sigma_3} , \Z / 2\Z)$ then $\tilde{U}$ in $\Sigma(L)$ is isotopic to $ T(2,4)$.
\end{lemma}

\begin{proof}
    For this computation, we will use the spectral sequence from $\AKh( L , \Z / 2\Z)$ to $\widehat{\HFK}(-\widetilde{U})$ and compute the Maslov gradings of the generators of $\AKh( L , \Z / 2\Z)$.
    
    From the construction of the spectral sequence from $\AKh( L , \Z / 2\Z)$ to $\widehat{\HFK}(-\widetilde{U})$ as an iterated mapping cone, it follows that in each $i$ grading on $\AKh( L , \Z / 2\Z)$ the relative Maslov grading of any two generators agrees with half the difference of the quantum gradings of the generators. It remains to relate the relative Maslov gradings for generators in $i$ grading $0$ and $1$ and then upgrade this information to an absolute Maslov grading.
    
    The induced differential $\partial_-$ giving the spectral sequence from $\AKh( L , \Z / 2\Z)$ to $\Kh( L , \Z / 2\Z)$ is part of the total differential on the iterated mapping cone induced by counting pseudo holomorphic polygons. Thus $\partial_-$ lowers the Maslov grading by one. This implies that for $\AKh( L , \Z / 2\Z)$, generators in the same $k$ grading live in the same relative Maslov grading. Since generators in the same $k$ grading or $2A$ grading also have the same Maslov grading, the spectral sequence to $\widehat{\HFK}(-\widetilde{U})$ collapses as all differentials preserve the $k$ grading or $2A$ grading and change the Maslov grading.
    
    To upgrade the above to a statement about the absolute Maslov grading, notice that there are only two generators survive in the spectral sequence to $\Kh( L , \Z / 2\Z)$, namely the generators that sit in the $k$ gradings $-4$ and $-2$. These generators must then be in Maslov gradings $0$ and $1$ respectively. From here we can pin down the Maslov gradings of the remaining six generators of $\AKh( L , \Z / 2\Z)$.
    
    The claim $\widehat{\HFK}(\widetilde{U}) \cong \widehat{\HFK}(T(2,4)) $ then follows from the fact that $\widehat{\HFK}(-\widetilde{U}) \cong (\widehat{\HFK}(\widetilde{U}))^* $ with the appropriate change in gradings. An application of Theorem~\ref{knotFloerT24} completes the proof of Lemma~\ref{4HFK}.

\end{proof}

With the previous lemmas we will complete the proof of Theorem~\ref{thm:4braid} by understanding the symmetries of $T(2,4)$. 

\begin{proof}[Proof of Theorem~\ref{thm:4braid}]
 
The complement of the torus link $T(2,4)$ is a Seifert fibered space with two exceptional fibers and the base space is an annulus. Any symmetry of $T(2,4)$ is isotopic to a symmetry which preserves the fibration~\cite{edmonds_group_1983}. The symmetries must also preserve the exceptional fibers so we can consider symmetries of the annulus with two marked points. There is an order two reflection that reverses the orientation of the annulus and an order two hyper elliptic involution that preserves orientation and interchanges the boundary components of the annulus. These two symmetries commute and generate the all symmetries of the annulus with two marked points. This shows that the symmetry group of $T(2,4)$ is $\mathbb{Z}/2\mathbb{Z} \oplus \mathbb{Z}/2\mathbb{Z} $ and that there is only a single symmetry of $T(2,4)$ which interchanges the components of the link and preserving the orientation of $S^3$. This is the symmetry induced by the hyper elliptic involution of the annulus.

If there is an annular link $L$ so that $T(2,4)$ is realized as $\tilde{U}$ in $\Sigma(L)$, then $T(2,4)$ has an order 2 orientation preserving symmetry that interchanges the components of $T(2,4)$ coming from the branched double covering deck transformation. Taking the quotient by this action and considering the image of the fixed point set of the involution in the complement of $T(2,4)$ recovers the annular link $L$. Because $T(2,4)$ only has one such symmetry up to isotopy the only possibility for the annular link $L$ is the braid closure $\widehat{\sigma_1 \sigma_2 \sigma_3}$.

\end{proof}

An argument of the same structure can be used to show that annular Khovanov homology detects the closure of the 4-braid $\sigma_1 \sigma_2\sigma_3\sigma_4\sigma_5$ in $A \times I$. We will first show that any braid closure with the same annular Khovaonov homology as $\widehat{\sigma_1 \sigma_2\sigma_3\sigma_4\sigma_5}$ must represent an unknot in $S^3$. We then show that the lifted braid axis is $T(2,6)$ using knot Floer homology. Then we complete the proof by understanding the symmetries of $T(2,6)$.

\begin{theorem}\label{thm:6braid}
    Let $L \subseteq A \times I \subseteq S^3$ be an annular link. If $\AKh( L , \Z / 2\Z) \cong \AKh( \widehat{\sigma_1 \sigma_2\sigma_3\sigma_4 \sigma_5} , \Z / 2\Z)$ then $L$ is isotopic to $ \widehat{\sigma_1 \sigma_2\sigma_3\sigma_4 \sigma_5}$ in $A \times I$.
\end{theorem}

We break the proof of Theorem~\ref{thm:6braid} up by first proving the following two lemmas. The proofs are similar to the proofs of Lemmas~\ref{4unknot} and~\ref{4HFK} respectively.

\begin{lemma}\label{6unknot}
     Let $L \subseteq A \times I \subseteq S^3$ be an annular link. If $\AKh( L , \Z / 2\Z) \cong \AKh( \widehat{\sigma_1 \sigma_2\sigma_3\sigma_4 \sigma_5} , \Z / 2\Z)$ then $L$ is an unknot in $S^3$.
\end{lemma}

\begin{proof}
    The proof proceeds similarly to the proof of Lemma~\ref{4unknot}.
    
    Because $L$ is a $6$-braid closure,  in grading $i=0$ $\AKh( L , \mathbb{C})$ must contain a weight $6$ irreducible $\mathfrak{sl}_2(\mathbb{C})$ representation, of dimension $7$. Thus $\AKh( L , \mathbb{C})$ must consist only of this representation in grading $i=0$ because $\AKh( L , \Z / 2\Z)$ has dimension $5$ in homological grading $0$. From this, we have that the generators in grading $i = 1$ for $\AKh( L , \Z / 2\Z)$ must also all correspond to generators of $\AKh( L , \mathbb{C})$ as well.
    
    From~\cite[Theorem~3.1]{BG}, we know that the differential $\partial_-$ on $\AKh( L , \mathbb{C})$ inducing the spectral sequence to $\Kh(L)$ must send the highest weight generator in $i = 0$ to the highest weight generator in $i = 1$. The action of $\partial_-$ is part of the action of $\mathfrak{sl}_2(\wedge)$ on $\AKh( L , \mathbb{C})$ and commutes up to sign with the lowering operator $f$~\cite[Theorem~1]{grigsby_annular_2018}. This means that the image of $\partial_-$ is the span of all generators in grading $i = 1$. Thus $\Kh(L)$ is dimension $2$ and $L$ is the unknot.
    
\end{proof}

	\begin{lemma}\label{6HFK}
     Let $L \subseteq A \times I \subseteq S^3$ be an annular link. If $\AKh( L , \Z / 2\Z) \cong \AKh( \widehat{\sigma_1 \sigma_2\sigma_3\sigma_4 \sigma_5} , \Z / 2\Z)$ then $\widetilde{U}$ in $\Sigma(L)$ is isotopic to $T(2,6)$.
\end{lemma}

\begin{proof}

 From the construction of the spectral sequence from $\AKh( L , \Z / 2\Z)$ to $\widehat{\HFK}(-\widetilde{U})$ as an iterated mapping cone, it follows that in each $i$ grading on $\AKh( L , \Z / 2\Z)$ the relative Maslov gradings of any two generators agrees with half the difference of the quantum gradings of the generators. It remains to relate the relative Maslov gradings of generators in $i$ grading $0$ and $1$ and upgraded this information to a statement concerning absolute Maslov grading.
    
    The induced differential $\partial_-$ that induces the spectral sequence from $\AKh( L , \Z / 2\Z)$ to $\Kh( L , \Z / 2\Z)$ is part of the total differential on the iterated mapping cone induced by counting pseudo holomorphic polygons. Thus $\partial_-$ must lower the Maslov grading by one. This implies that for $\AKh( L , \Z / 2\Z)$, generators in the same $k$ grading live in the same relative Maslov grading. Since the generators in the same $k$ or $2A$ grading have the same Maslov grading, the spectral sequence to $\widehat{\HFK}(-\widetilde{U})$ has collapsed since all of the differentials in this spectral sequence preserve the $k$ or $2A$ grading and change the Maslov grading.
    
    To upgrade to the above to a statement about the absolute Maslov grading, notice that there are only two generators survive in the spectral sequence to $\Kh( L , \Z / 2\Z)$, namely the generators in $k$ gradings $-4$ and $-2$. These two generators must therefore have Maslov gradings $0$ and $1$ respectively. From here we can pin down the Maslov gradings of the other ten generators of $\AKh( L , \Z / 2\Z)$.
    
    The claim that $\widehat{\HFK}(\widetilde{U}) \cong \widehat{\HFK}(T(2,6)) $ then follows from the fact that $\widehat{\HFK}(-\widetilde{U}) \cong (\widehat{\HFK}(\widetilde{U}))^* $ with the appropriate change in gradings. An application of Theorem~\ref{HFKT26} completes the proof.
    
\end{proof}

With the previous lemmas we will complete the proof of Theorem~\ref{thm:4braid} by understanding the symmetries of $T(2,6)$. 

\begin{proof}[Proof of Theorem~\ref{thm:6braid}]

The complement of the torus link $T(2,6)$ is a Seifert fibered space with two exceptional fibers and the base space is an annulus. Any symmetry of $T(2,6)$ is isotopic to a symmetry which preserves the fibration~\cite{edmonds_group_1983}. The symmetries must also preserve the exceptional fibers so we can consider symmetries of the annulus with two marked points. There is an order two reflection that reverses the orientation of the annulus and an order two hyper elliptic involution that preserves orientation and interchanges the boundary components of the annulus. These two symmetries commute and generate the all symmetries of the annulus with two marked points. This shows that the symmetry group of $T(2,6)$ is $\mathbb{Z}/2\mathbb{Z} \oplus \mathbb{Z}/2\mathbb{Z} $ and that there is only a single symmetry of $T(2,6)$ which interchanges the components of the link and preserving the orientation of $S^3$. This is the symmetry induced by the hyper elliptic involution of the annulus.

If there is an annular link $L$ so that $T(2,6)$ is realized as $\tilde{U}$ in $\Sigma(L)$, then $T(2,6)$ has an order 2 orientation preserving symmetry that interchanges the components of $T(2,6)$ coming from the branched double covering deck transformation. Taking the quotient by this action and considering the image of the fixed point set of the involution in the complement of $T(2,6)$ recovers the annular link $L$. Because $T(2,6)$ only has one such symmetry up to isotopy the only possibility for the annular link $L$ is the braid closure $\widehat{\sigma_1 \sigma_2 \sigma_3\sigma_4 \sigma_5}$.

\end{proof}

\begin{remark}
	Similar techniques could be applied to show that annular Khovanov homology detects the closure of the braid $\sigma_1 \in B_2$. However, the fact that annular Khovanov homology detects this braid closure already follows from the fact that annular Khovanov homology detects braid closures and the braid index~\cite[Corollary~1.2]{grigsby_sutured_2014}~\cite[Corollary~8.4]{xie_instanton_2019}, combined with the computations of the annular Khovanov homologies of all 2-braid closures~\cite[Proposition~15]{grigsby_annular_2018}.
	\end{remark}

	\begin{remark}
	  The reader may wonder why the arguments in Theorem~\ref{thm:4braid} and~\ref{thm:6braid} involve the symmetries of torus links but in Theorem~\ref{detect} these arguments aren't needed. The difference is that the Birman-Hilden correspondence induces an isomorphism between the 3-stranded braid and the mapping class group of the once punctured torus but for higher braid groups it only gives an injective map. If more was known about the Birman-Hilden correspondence then it might be possible to avoid these arguements about symmetries. In the case of 4-braids a question is if $BH(\beta)$ is conjugate to $BH(\sigma_1 \sigma_2\sigma_3)$ in $\Mod(S_1^2)$ then is $\beta$ conjugate to $\sigma_1 \sigma_2\sigma_3$ in $B_4$? If the answer to this question is yes then it would be possible to prove Theorem~\ref{thm:4braid} without considering the symmetries of $T(2,4)$. The same idea would be true for Theorem~\ref{thm:6braid} and a similar question for the Birman-Hilden correspondence for $B_6$.
	\end{remark}

	\end{section}

\bibliographystyle{plain}
\bibliography{name}

\end{document}